\documentclass[12pt,a4paper]{article}
\usepackage{amssymb,amsmath,amsthm}
\usepackage[english]{babel}
\usepackage{t1enc}
\usepackage[latin2]{inputenc}
\usepackage{epsfig}
\usepackage{subfig}
\usepackage{epstopdf}

\newtheorem{thm}{Theorem}[section]
\newtheorem{cor}[thm]{Corollary}
\newtheorem{defi}[thm]{Definition}
\newtheorem{lem}[thm]{Lemma}
\newtheorem{prop}[thm]{Proposition}

\newtheorem{prob}[thm]{Problem}
\newtheorem{obs}[thm]{Observation}
\newtheorem{ex}[thm]{Example}

\theoremstyle{remark}
\newtheorem{remark}[thm]{Remark}

\usepackage{indentfirst}
\def\R{\mathbb{R}}
\def\HH{\mathcal {H}}

\def\FF{\mathcal {F}}
\def\UU{\mathcal {U}}
\def\DD{\mathcal {D}}

\makeatletter
\def\blfootnote{\gdef\@thefnmark{}\@footnotetext}
\makeatother

\begin{document}

\title{An abstract approach to polychromatic coloring: shallow hitting sets in ABA-free hypergraphs and pseudohalfplanes}
\author{Bal\'azs Keszegh\thanks{Research supported by the National Research, Development and Innovation Office -- NKFIH under the grant K 116769 and by the Lend\"ulet program of the Hungarian Academy of Sciences (MTA), under grant number LP2017-19/2017.}\\
\and D\"om\"ot\"or P\'alv\"olgyi\thanks{Research supported by the Marie Sk\l odowska-Curie action of the EU, under grant IF 660400 and by the Lend\"ulet program of the Hungarian Academy of Sciences (MTA), under grant number LP2017-19/2017.}
}

\maketitle


\begin{abstract}
The goal of this paper is to give a new, abstract approach to cover-decomposition and polychromatic colorings using hypergraphs on ordered vertex sets. 
We introduce an abstract version of a framework by Smorodinsky and Yuditsky, used for polychromatic coloring halfplanes, and apply it to so-called {\em ABA-free hypergraphs}, which are a generalization of {\em interval graphs}.
Using our methods, we prove that $(2k-1)$-uniform ABA-free hypergraphs have a polychromatic $k$-coloring, a problem posed by the second author.
We also prove the same for hypergraphs defined on a point set by pseudohalfplanes.
These results are best possible.
We could only prove slightly weaker results for dual hypergraphs defined by pseudohalfplanes, and for hypergraphs defined by pseudohemispheres.
We also introduce another new notion that seems to be important for investigating polychromatic colorings and $\epsilon$-nets, {\em shallow hitting sets}.
We show that all the above hypergraphs have shallow hitting sets, if their hyperedges are containment-free.
\end{abstract}


\section{Introduction}

The study of proper and polychromatic colorings of geometric hypergraphs has attracted much attention, not only because this is a very basic and natural theoretical problem but also because such problems often have important applications.
One such application area is resource allocation, e.g., battery consumption in sensor networks.
Moreover, the coloring of geometric shapes in the plane is related to the problems of cover-decomposability, conflict-free colorings and $\epsilon$-nets; these problems have applications in sensor networks and frequency assignment as well as other areas. For surveys on these and related problems see \cite{PPT14,cfsurvey}.

In a (primal) {\em geometric hypergraph polychromatic coloring} problem, we are given a natural number $k$, a set of points and a collection of regions in $\R^d$, and our goal is to $k$-color the points such that every region that contains at least $m(k)$ points contains a point of every color, where $m$ is some function that we try to minimize.
We call such a coloring a {\em polychromatic $k$-coloring}. 
In a {\em dual} geometric hypergraph polychromatic coloring problem, our goal is to $k$-color the regions such that every point which is contained in at least $m(k)$ regions is contained in a region of every color.
In other words, in the dual version our goal is to {\em decompose} an $m(k)$-fold covering of some point set into $k$ coverings.
The primal and the dual versions are equivalent if the underlying regions are the translates of some fixed set.
For the proof of this statement and an extensive survey of results related to {\em cover-decomposition}, see e.g., \cite{PPT14}.
Below we mention some of these results, stated in the equivalent primal form.

The most general result about translates of polygons is that given a fixed convex polygon, there exists a $c$ (that depends only on the polygon) such that any {\em finite} point set has a  polychromatic $k$-coloring such that any translate of the fixed convex polygon that contains at least $m(k)=c\cdot k$ points  contains a point of every color \cite{GV10}.
Non-convex polygons for which such a finite $m(k)$ (for any $k\ge 2$) exists have been classified \cite{P10,PT10}.

As it was shown recently \cite{P14}, there is no such finite $m(2)$ for convex sets with a smooth boundary, e.g., for the translates of a disc.
However, it was also shown in the same paper that for the translates of any {\em unbounded} convex set $m(2)=3$ is sufficient.
In this paper we extend this result to every $k$, showing that $m(k)=2k-1$ is an optimal function for unbounded convex sets.
Our proof is an abstraction of a method developed by Smorodinsky and Yuditsky \cite{SY12}.

For homothets of a given shape the primal and dual problems are not equivalent.
For homothets of a triangle (a case closely related to the case of translates of octants \cite{KP12,KP14octants}), there are several results, the current best are $m(k)= O(k^{4.09})$ in the primal version \cite{CKMU13,KP15} and $m(k)= O(k^{5.09})$ in the dual version \cite{CKMU14,KP15}.
For the homothets of other convex polygons, in the dual case there is no finite $m(2)$ \cite{K14}, and in the primal case only conditional results are known \cite{KP14}, namely, that the existence of a finite $m(2)$ implies the existence of an $m(k)$ that grows at most polynomially in $k$.
In fact, it is even possible that for {\em any} polychromatic coloring problem $m(k)=O(k\cdot m(2))$.

For other shapes, cover-decomposability has been studied less, in these cases the investigation of polychromatic-colorings is motivated rather by conflict-free colorings or $\epsilon$-nets.
Most closely related to our paper, coloring halfplanes for small values were investigated in \cite{wcf,K12,F11}, and polychromatic $k$-colorings in \cite{SY12}.
We generalize all the (primal and dual) results of the latter paper to pseudohalfplanes, answering a question left open by the authors.\footnote{Personal communication, Shakhar Smorodinsky.}
Note that translates of an unbounded convex set form a set of pseudohalfplanes, thus the above mentioned result about unbounded convex sets is a special case of this generalization to pseudohalfplanes. 

Axis-parallel rectangles are usually investigated from the $\epsilon$-net point of view (e.g., \cite{CHPSZT08,PaTa10}), for which the coloring function $f$ is not independent of the number of points/regions.
Motivated by these, bottomless rectangles are regarded for small values in \cite{wcf,K12} and polychromatic $k$-colorings in \cite{A13}.
In this paper we place bottomless rectangles in our abstract context and pose some further problems about them.

Besides generalizing earlier results, our contribution is a more abstract approach to the above problems.
Namely, we introduce the notion of {\em ABA-free hypergraphs} (see Definition~\ref{def:ABA}), {\em shallow hitting sets} (see Definition~\ref{defi:shallow}) and {\em balanced polychromatic colorings} (see Definition~\ref{def:balanced}), and discuss their relevance.

In such a coloring context ABA-free hypergraphs were first defined in \cite{P14} under the name {\em special shift-chains}, as they are a special case of {\em shift-chains} introduced in \cite{PhD}. However, such families were regarded earlier, similarly motivated by their geometric interpretations. Namely, in \cite{kintersect} they consider $k$-intersecting $x$-monotone curves and $k$-intersecting families of ($0$-$1$-)vectors. With our definitions $1$-intersecting $x$-monotone curves are exactly pseudolines, while $1$-intersecting (resp.\ $2$-intersecting) families of vectors are exactly the families of characteristic vectors of ABA-free (resp.\ ABAB-free, see Definition  \ref{def:abab}) families of sets. The primary interest of \cite{kintersect} lies in determining the maximum size of $l$-uniform $k$-intersecting families of vectors on $n$ coordinates.
\subsection{Definitions and statements of main results}


\begin{defi}\label{def:ABA}
A hypergraph $\mathcal H$ with an ordered vertex set is called {\em ABA-free} if $H$ does not contain two hyperedges $A$ and $B$ for which there are three vertices $x<y<z$ such that $x,z\in A\setminus B$ and $y\in B\setminus A$.

A hypergraph with an unordered vertex set is ABA-free if its vertices have an ordering with which the hypergraph is ABA-free.\footnote{While it might seem that using the same notion for ordered and unordered hypergraphs leads to confusion as by forgetting the ordering of an ordered hypergraph it might become ABA-free, from the context it will always be perfectly clear what we mean.}
\end{defi}

\begin{ex}\label{jamegezisegypelda}
An {\em interval hypergraph} is a hypergraph whose vertices are some points of $\R$, and its hyperedges are some intervals from $\R$, with the incidences preserved.
\end{ex}

\begin{ex}[\cite{P14}]\label{egyikpelda}
Let $S$ be a set of points in the plane with different $x$-coordi\-na\-tes and let $C$ be a convex set that contains a vertical halfline.
Define a hypergraph $\mathcal H$ whose vertex set is the $x$-coordinates of the points of $S$.
A set of numbers $X$ is a hyperedge of $\mathcal H$ if there is a translate of $C$ such that the 
$x$-coordinates of the points of $S$ contained in the translate are exactly $X$.
The hypergraph $\mathcal H$ defined this way is ABA-free.
\end{ex}

\begin{ex}\label{masikpelda}
Let $S$ be a set of points in the plane in general position.
Define a hypergraph $\mathcal H$ whose vertex set is the $x$-coordinates of the points of $S$.
A set of numbers $X$ is a hyperedge of $\mathcal H$ if there is a positive halfplane $H$ (i.e., that contains a vertical positive halfline) such that the set of $x$-coordinates of the points of $S$ contained in $H$ is $X$. 
The hypergraph $\mathcal H$ defined this way is ABA-free.
\end{ex}

The above examples show how to reduce geometric problems to abstract problems about ABA-free hypergraphs.
Observe that given an $S$, by choosing an appropriately big parabola, any hyperedge defined by a positive halfplane as in Example~\ref{masikpelda} is also defined by some translate of the big parabola as in Example~\ref{egyikpelda}, thus Example~\ref{egyikpelda} is more general than Example~\ref{masikpelda}, and it is easy to see that both are more general than Example~\ref{jamegezisegypelda}.
Even more, as we will see later in Section~\ref{sec:pseudo}, finite ABA-free hypergraphs have an equivalent geometric representation with graphic pseudoline arrangements (here hyperedges are defined by the regions above the pseudolines, for the definitions and details see Section~\ref{sec:pseudo}) and both translates of the boundary of an unbounded convex set and lines in the plane form graphic pseudoline arrangements, showing again that the above examples are special cases of ABA-free hypergraphs.

To study polychromatic coloring problems, we also introduce the following definition, which is implicitly used in \cite{SY12}, but deserves to be defined explicitly as it seems to be important in the study of polychromatic colorings.

\begin{defi}\label{defi:shallow}
A set $R$ is a {\em $c$-shallow hitting set} of the hypergraph $\mathcal H$ if for every $H\in \mathcal H$ we have $1\le |R\cap H|\le c$.
\end{defi}

Actually, almost all our results are based on shallow hitting sets.

Our main results and the organization of the rest of this paper are as follows.

In Section~\ref{sec:ABA} we prove 
(following closely the ideas of Smorodinsky and Yuditsky~\cite{SY12}) that every $(2k-1)$-uniform ABA-free hypergraph has a polychromatic coloring with $k$ colors. We then observe that the dual of this problem is equivalent to the primal, which implies that the hyperedges of every $(2k-1)$-uniform ABA-free hypergraph can be colored with $k$ colors, such that if a vertex $v$ is in a subfamily $\HH_v$ of at least $m(k)=2k-1$ of the hyperedges of $\HH$, then $\HH_v$ contains a hyperedge from each of the $k$ color classes. 

In Section~\ref{sec:pseudo} we give an abstract equivalent definition (using ABA-free hypergraphs) of hypergraphs defined by pseudohalfplanes, and we prove 
that given a finite set of points $S$ and a pseudohalfplane arrangement $\HH$, we can $k$-color $S$ such that any pseudohalfplane in $\HH$ that contains at least $m(k)=2k-1$ points of $S$ contains all $k$ colors. 
Both results are sharp. Note that these results imply the same for hypergraphs defined by unbounded convex sets.

In Section~\ref{sec:signed} we discuss dual and other versions of the problem.
For example we prove 
that given a pseudohalfplane arrangement $\HH$, we can $k$-color $\HH$ such that if a point $p$ belongs to a subfamily $\HH_p$ of at least $m(k)=3k-2$ of the pseudohalfplanes of $\HH$, then $\HH_p$ contains a pseudohalfplane from each of the $k$ color classes.
This result might not be sharp, the best known lower bound for $m(k)$ is $2k-1$ \cite{SY12}.

In Section~\ref{sec:bless}, we discuss ABAB-free hypergraphs and related problems.
We also discuss consequences about  $\epsilon$-nets on pseudohalfplanes in Appendix~\ref{app:epsnet}.


\medskip
We denote the symmetric difference of two sets, $A$ and $B$, by $A\Delta B$, the complement of a hyperedge $F$ by $\bar F$ and for a family $\FF$ we use $\bar\FF=\{\bar F \mid F\in\FF\}$.
We will suppose (unless stated otherwise) that all hypergraphs and point sets are finite, and denote the smallest (resp.\ largest) element of an ordered set $H$ by $\min(H)$ (resp.\ $\max(H)$).

\section{ABA-free hypergraphs and the general coloring algorithm}\label{sec:ABA}

Suppose we are given an ABA-free hypergraph $\mathcal H$ on $n$ vertices.
As the hypergraph is ABA-free, for any pair of sets $A, B \in \mathcal H$ either there are 
$a<b$ such that $a\in A\setminus B$ and $b\in B\setminus A$, or there are 
$b<a$ such that $a\in A\setminus B$ and $b\in B\setminus A$, or none of them, but not both as that would contradict ABA-freeness.

Define $A<B$ if and only if there are $a<b$ such that $a\in A\setminus B$ and $b\in B\setminus A$, and $A\le B$ if and only if either $A=B$ (as sets) or $A<B$.
By the above, this is well-defined, and below we show that it gives a partial ordering of the sets.

\begin{obs} If $A<B$ and $a\in A\setminus B$, then there is a $b\in B\setminus A$ such that $b>a$.
\end{obs}

\begin{prop}\label{prop2.2} If $A<B$ and $B<C$, then $A<C$.
\end{prop}
\begin{proof} Take an $a\in A\setminus B$.
If $a\notin C$, then take a $b\in B\setminus A$.
If $b\in C$, then $A<C$ and we are done.
Otherwise, there has to be a $c>b$ such that $c\in C\setminus B$.
If $c\in A$, then $a<b<c$ forms a forbidden sequence for $A$ and $B$, thus $c\notin A$.
Then by definition $a$ and $c$ show that $A<C$. 

If $a\in C$, then also $a\in C\setminus B$, thus there has to be a $b_1<a$ such that $b_1\in B\setminus C$.
As $a\in A\setminus B$ and $A<B$, we also have $b_1\in A$ and so $b_1\in A\setminus C$.
There also has to be a $b_2>a$ such that $b_2\in B\setminus A$.
If $b_2\notin C$, then $b_1<a<b_2$ forms a forbidden sequence for $B$ and $C$.
Thus $b_2\in C\setminus A$, and by definition $b_1$ and $b_2$ show that $A<C$.
\end{proof}

We proceed with another definition.

\begin{defi}
A vertex $a$ is {\em skippable} if there exists an $A\in \HH$ such that $\min(A)< a < \max(A)$ and $a\notin A$.
In this case we say that $A$ {\em skips} $a$. 
A vertex $a$ is {\em unskippable} if there is no such $A$.
\end{defi}

\begin{obs}
If a vertex $a$ is unskippable in some ABA-free hypergraph $\HH$, then after adding the one-element hyperedge $\{a\}$ to $\HH$, it remains ABA-free.
\end{obs}

Note that the following two lemmas show that the unskippable vertices of an ABA-free hypergraph behave with respect to hyperedges similarly to how the vertices on the convex hull of a point set behave with respect to halfplanes.
These two lemmas make it possible to use the framework of \cite{SY12} on ABA-free hypergraphs.

\begin{lem}\label{lem:unskippable}
If $\mathcal H$ is ABA-free, then every $A\in \mathcal H$ contains an unskippable vertex.
\end{lem} 
\begin{remark}
Note that finiteness (recall that we have supposed that all our hypergraphs are finite) is needed, as the hypergraph whose vertex set is $\mathbb Z$ and hyperedge set is $\{ \mathbb Z\setminus \{n\}\mid n\in \mathbb Z\}$ is ABA-free without unskippable vertices.
\end{remark}
\begin{proof}[Proof of Lemma~\ref{lem:unskippable}] 
Take an arbitrary set $A\in \HH$, suppose that it does not contain an unskippable vertex, we will reach a contradiction.
Call $a\in A$ {\em rightskippable} if there is a $B\in \HH$ rightskipping $a$, that is for which $a\in A\setminus B$ and
there are $b_1, b_2\in B$ such that $b_1<a<b_2$ where $b_2\in B\setminus A$. 

If $A$ contains no unskippable vertex, $\max(A)$ must be rightskippable (any set skipping $\max(A)$ must also rightskip $\max(A)$). Also, $\min(A)$ cannot be rightskippable, as otherwise $A$ and the set $B$ rightskipping $\min(A)$ would violate ABA-freeness (we would get $b_1<\min(A)<b_2$ where $b_1,b_2\in B\setminus A, \min(A)\in A\setminus B$).
Therefore we can take the largest $a\in A$ that is not rightskippable.
By the assumption, it is skipped by a set, call it $B$, i.e., $b_1<a<b_2$ where $b_1, b_2\in B \not\owns a$.
Moreover, suppose without loss of generality that $b_2$ is the smallest element of $B$ which is bigger than $a$.
Since $a$ is not rightskippable, $b_2\in A$ must also hold. 
As $b_2\in A$ is rightskippable, there is a $C$ such that $c_1<b_2<c_2$ where $c_1, c_2\in C$ and $b_2\notin C, c_2\notin A$.
Without loss of generality, suppose that $c_1$ is the largest element of $C$ which is smaller than $b_2$. 
If $c_1<a$, then $C$ would rightskip $a$, a contradiction.
Thus, $b_1<a\le c_1,$ and from the choice of $b_2$ we conclude that $c_1\notin B$.
As $c_2\notin A$, also $c_2\notin B$, otherwise $B$ would rightskip $a$.
Putting all together, we get $c_1<b_2<c_2$, thus $B$ and $C$ contradict ABA-freeness.
\end{proof}

\begin{defi}
A hypergraph is called {\em containment-free} if none of its hyperedges contains another hyperedge.\footnote{Equivalently, the hyperedges form an {\em antichain}. This property is also called {\em Sperner}.}
A hypergraph ${\cal H}'$ is a {\em subhypergraph} of a hypergraph $\cal H$ on vertex set $S$ if we can get ${\cal H}'$ by taking a subset $S'\subset S$ as its vertex set and the family of the hyperedges of ${\cal H}'$ is a subfamily of the hyperedges of $\cal H$ restricted to $S'$.
We call a hypergraph property $\cal P$ {\em hereditary} if for every hypergraph $\cal H$ that has property $\cal P$, all of its subhypergraphs also have property $\cal P$.
\end{defi}

\begin{obs}\label{induced}
	ABA-freeness is a hereditary property.
\end{obs}

We further assume in the rest of the paper that our hypergraphs are nonempty in the sense that they contain at least one hyperedge which is not the empty set.
Notice that for an ABA-free containment-free hypergraph the ordering $<$ of its sets is a total order, i.e., any two hyperedges are comparable.


\begin{lem}\label{lem:spshallow}
If $\mathcal H$ is ABA-free and containment-free, then any minimal hitting set of $\mathcal H$ that contains only unskippable vertices is $2$-shallow.
\end{lem}
\begin{proof}
Let $R$ be a minimal (for containment) hitting set of unskippable vertices.
Assume to the contrary that there exists a set $A$ such that $|A\cap R|\ge 3$.
Let $l=\min(A\cap R)$ and $r=\max(A\cap R)$. 
There exists a third vertex $l<a<r$ in $A\cap R$.
We claim that $R'=R\setminus \{a\}$ hits all sets of $\mathcal H$, contradicting the minimality of $R$.
Assume on the contrary that $R'$ is disjoint from some $B\in \mathcal H$.
As $R$ must hit $B$, we have $R\cap B=\{a\}$.
If there is a $b\in B\setminus A$ such that $l<b<r$, that would contradict the ABA-free property.
If there is a $b\in B$ such that $b<l<a$ or $a<r<b$, that would contradict that $l$ and $r$ are unskippable.
Thus $B\subset A$, contradicting that $\mathcal H$ is containment-free.
\end{proof}

\begin{lem}\label{lem:abashallow}
Every containment-free ABA-free hypergraph has a $2$-shallow hitting set.
\end{lem}
\begin{proof}
Given a containment-free ABA-free hypergraph, take the set of all unskippable vertices, it is a hitting set by Lemma~\ref{lem:unskippable}. Then we can delete vertices from this set until it becomes a minimal hitting set, which is $2$-shallow by Lemma \ref{lem:spshallow}.
\end{proof}

Now we present an abstract and generalized version of the framework of \cite{SY12} to give polychromatic $k$-colorings of hypergraphs. 

\begin{thm}\label{algo:sy}
	Assume that $\cal P$ is a hereditary hypergraph property such that every containment-free hypergraph with property $\cal P$ has a $c$-shallow hitting set.
	Then every hypergraph $\cal H$ with hyperedges of size at least $ck-(c-1)$ that has property $\cal P$ admits a polychromatic $k$-coloring, i.e., a coloring of its vertices with $k$ colors such that every hyperedge of $\cal H$ contains vertices of all $k$ colors.
\end{thm}
\begin{proof}
 We present an algorithm that gives a polychromatic $k$-coloring.
First, we repeat $k-1$ times ($i=1,\dots ,k-1$) the {\em general step} of the algorithm:
	
	At the beginning of step $i$ we have a hypergraph $\cal H$ with hyperedges of size at least $ck-ci+1$ that has property $\cal P$.
	If any hyperedge contains another, then delete the bigger hyperedge.
	Repeat this until no hyperedge contains another, thus making our hypergraph containment-free.
	Next, take a $c$-shallow hitting set (using our assumptions), and color its vertices with the $i$-th color.
	Delete these vertices from $\cal H$ (the hyperedges of the new hypergraph are the ones induced by the remaining vertices).
	As $\cal P$ is hereditary, the new hypergraph also has property $\cal P$ and we can proceed to the next step.
	
	After $k-1$ iterations of the above, we are left with a $1$-uniform hypergraph whose vertices we can color with the $k$-th color.
\end{proof}

First, we use this algorithm to give a polychromatic $k$-coloring of the vertices of an ABA-free hypergraph with hyperedges of size at least $2k-1$.

\begin{thm}\label{thm:ABA}
	Given an ABA-free ${\mathcal H}$ we can color its vertices with $k$ colors such that every $A\in {\mathcal H}$ whose size is at least $2k-1$ contains all $k$ colors.
\end{thm}
\begin{proof}
By Observation~\ref{induced} ABA-freeness is a hereditary property. Together with Lemma \ref{lem:abashallow} we get that all the assumptions of Theorem~\ref{algo:sy} with $c=2$ hold for ABA-free hypergraphs with hyperedges of size at least $2k-1$ and thus we get a required $k$-coloring. 
\end{proof}

Notice that the above theorem is sharp, as taking $\mathcal H$ to be all subsets of size $2k-2$ from $2k-1$ vertices, in any coloring of the vertices, one color must occur at most once and is thus missed by some hyperedge.

We state another corollary of Lemma~\ref{lem:unskippable} that we need later.
Before that, we need another simple claim.

\begin{prop}\label{prop:insert} If we insert a new vertex, $v$, somewhere into the (ordered) vertex set of an ABA-free hypergraph, $\HH$, and add $v$ to every hyperedge that contains a vertex before and another vertex after $v$, then we get an ABA-free hypergraph.
\end{prop}
\begin{proof} We show that if in the new hypergraph, $\HH'$, two hyperedges $A'$ and $B'$ violate ABA-freeness, then we can find two hyperedges $A$ and $B$ in the original hypergraph, $\HH$, that also violate ABA-freeness, which would be a contradiction.
We define $A=A'\setminus\{v\}$ and $B=B'\setminus\{v\}$.
If both $A'$ and $B'$ contain or do not contain $v$, then by definition $A$ and $B$ also violate the condition.
If, say, $v\in A'$ and $v\notin B'$, then without loss of generality we can suppose that all the vertices of $B=B'$ are before $v$.
This means that if there are $x<y<z$ such that $x,z\in A'\setminus B'$ and $y\in B'\setminus A'$, then necessarily $v=z$.
But as $A'$ has an element $z'$ that is bigger than $v$, we have $x,z'\in A\setminus B$ and $y\in B\setminus A$, a contradiction.
\end{proof}

\begin{lem}\label{prop:slack}
If $\mathcal H$ is ABA-free, $A\in\HH $, then there is a vertex $a\in A$ such that $\HH\cup\{A\setminus\{a\}\}$ is also ABA-free.
\end{lem}
\begin{proof}
If $|A|=1$, then trivially $\mathcal H$ can be extended with $\emptyset$.
If $|A|>1$, then we proceed by induction on the size of $A$.
Using Lemma~\ref{lem:unskippable}, there is an unskippable vertex $v\in A$.
Delete this vertex from $\HH$ to obtain some ABA-free $\HH_v$ and let $A_v=A\setminus \{v\}$. 
Using induction on $A_v$, there is an $A_v'=A_v\setminus\{a\}$ such that $\HH_v\cup\{A_v'\}$ is also ABA-free.
We claim that with $A'=A_v'\cup \{v\}=A\setminus\{a\}$, the family $\HH\cup\{A'\}$ is also ABA-free.

Notice that adding back $v$ to $\HH_v$ is very similar to the operation of Proposition~\ref{prop:insert}, as $v$ is unskippable in $\HH$.
The only difference is that we might also have to add it to some further hyperedges, ending in or starting at $v$.
But a hyperedge that contains $v$ cannot violate the ABA-free condition with $A'$, since it also contains $v$, so the corresponding hyperedges in $\HH_v$ would also violate the ABA-free condition.
\end{proof}

Notice that with the repeated application of Lemma~\ref{prop:slack} we can extend any ABA-free hypergraph, such that in any set $A$ there is a vertex $a$ for which $\{a\}$ is a singleton hyperedge, implying that $a$ is unskippable in $A$.
Thus in fact Lemma~\ref{prop:slack} is equivalent to Lemma~\ref{lem:unskippable}.
Moreover, in Section~\ref{sec:pseudo}, in the more general context of pseudohalfplanes, it will be the abstract equivalent of a known and important property of pseudohalfplanes.

We prove another interesting property of ABA-free hypergraphs before which we need the following definition.

\begin{defi}\label{def:dual} The dual of a hypergraph $\mathcal H$, denoted by $\mathcal H^*$, is such that its vertices are the hyperedges of $\mathcal H$ and its hyperedges are the vertices of $\mathcal H$ with the same incidences as in $\mathcal H$.
\end{defi}

\begin{prop}\label{prop:dual} If ${\mathcal H}$ is ABA-free, then its dual ${\mathcal H^*}$ is also ABA-free (with respect to some ordering of its vertices).
\end{prop}
\begin{proof} 
Take the partial order ``$<$'' of the hyperedges of $\mathcal H$ and extend this arbitrarily to a total order $<^*$.
We claim that ${\mathcal H^*}$ is ABA-free if its vertices are ordered with respect to $<^*$.
To check the condition, suppose for a contradiction that $H_x<^*H_y<^*H_z$ and $a\in (H_x\cap H_z)\setminus H_y$ and $b\in H_y\setminus (H_x\cup H_z)$.
Without loss of generality, suppose that $a<b$.
But in this case $H_z<H_y$ holds, contradicting $H_y<^*H_z$.
\end{proof}

\begin{cor}
The hyperedges of every ABA-free hypergraph can be colored with $k$ colors, such that if a vertex $v$ is in a subfamily $\HH_v$ of at least $m(k)=2k-1$ of the hyperedges of $\HH$, then $\HH_v$ contains a hyperedge from each of the $k$ color classes.
\end{cor}

\begin{cor}
Any $(2k-1)$-fold covering of a finite point set with the translates of an unbounded convex planar set is decomposable into $k$ coverings.
\end{cor}

In fact, there is a slightly different proof for Proposition \ref{prop:dual}. For that we give an equivalent definition of ABA-free hypergraphs in relation to their incidence matrices, which will be useful also for other purposes later. In an \emph{incidence matrix} of a hypergraph $\HH$, rows correspond to the vertices of $\HH$, columns correspond to the hyperedges of $\HH$. An entry is $1$ if the hyperedge corresponding to the column contains the vertex corresponding to the row, and $0$ otherwise. Note that this is not unique as we can order the rows and columns arbitrarily. We say that a matrix $M$ \emph{contains} another matrix $P$ if $P$ is a submatrix of $M$. If $M$ does not contain $P$, then it is called \emph{$P$-free}. 

\begin{thm}\label{thm:matrix}
	Given a hypergraph $\HH$, the following are equivalent:
	\begin{itemize}
		\item[(a)] $\HH$ is an ABA-free hypergraph,
		\item[(b)] there is a permutation of the rows of the incidence matrix of $\HH$ such that the matrix becomes  
		$\begin{bmatrix} 
		0 & 1  \\
		1 & 0  \\
		0 & 1  	
		\end{bmatrix}$-free and 
				$\begin{bmatrix} 
		1 & 0  \\
		0 & 1  \\
		1 & 0  	
		\end{bmatrix}$-free,
		\item[(c)] there is a permutation of the rows and columns of the incidence matrix of $\HH$ such that the matrix becomes  
		$\begin{bmatrix} 
		0 & 1  \\
		1 & 0 
		\end{bmatrix}$-free.
	\end{itemize}
\end{thm}	

\begin{proof}
	First, ordering the vertices of the hypergraph corresponds to permuting the rows of its incidence matrix. Thus, the equivalance of $(a)$ and $(b)$ follows from the definition of ABA-free hypergraphs. 
		
	To prove $(c)\rightarrow(b)$, suppose $(b)$ is false, i.e., that in any permutation of the rows of the indicence matrix of $\HH$ there is an occurrence of one of the two matrices forbidden in $(b)$. In any permutation of the two columns of these two matrices forbidden in $(b)$, we get back one of these two matrices, both of which contains a copy of $\begin{bmatrix} 
	0 & 1  \\
	1 & 0 
	\end{bmatrix}$. Thus by any permutation of the rows and columns of the incidence matrix of $\HH$, we get a matrix that contains $\begin{bmatrix} 
	0 & 1  \\
	1 & 0 
	\end{bmatrix}$. Thus we can conclude that $\neg(b)\rightarrow \neg(c)$, which is the contrapositive of $(c)\rightarrow(b)$.
	
	 Finally, extending to a complete order the partial ordering  ``$<$'' defined on the hyperedges at the beginning of this section, Proposition \ref{prop2.2} implies that by permuting the columns according to any extension of this order ``$<$'' of the hyperedges we get a $\begin{bmatrix} 
	 0 & 1  \\
	 1 & 0 
	 \end{bmatrix}$-free matrix, and thus
	 $(b)\rightarrow (c)$.
\end{proof}

Now observe that in Theorem \ref{thm:matrix} the property in $(c)$ holds for an incidence matrix if and only if it holds for its transpose  (as the forbidden matrix in $(c)$ is its own transpose). Taking the transpose of an incidence matrix in terms of the hypergraph means taking the dual of the hypergraph, thus Proposition \ref{prop:dual} follows.

\section{Pseudohalfplanes}\label{sec:pseudo}
Here we extend a result of Smorodinsky and Yuditsky~\cite{SY12}.
A {\em pseudoline arrangement} is a finite collection of simple curves in the plane such that each curve cuts the plane into two components (i.e., both endpoints of each curve are at infinity) and any two of the curves are either disjoint or intersect once, and in the intersection point they {\em cross}, meaning that any finite perturbation of the curves contains an intersection point.
We also suppose that the curves are in general position, i.e., no three curves have a common point.
Some well-known results about pseudoline arrangements are collected in Appendix~\ref{app:facts}, which can be found in \cite{B99}.
We also recommend \cite{D07} where generalizations of classical theorems are proved for {\em topological affine planes}.
From these, it follows that the hypergraphs defined by points contained in pseudohalfplanes are exactly the ones that have the following structure.

\begin{defi}\label{def:pseudo}
A hypergraph $\HH$ on an ordered set of points $S$ is called a  {\em pseudo\-halfplane-hypergraph} if there exists an ABA-free hypergraph $\FF$ on $S$ such that $\HH\subset\FF\cup \bar{\FF}$.
\end{defi}

Note that  $\bar{\FF}$ is also ABA-free with the same ordering of the points.
We refer to the hyperedges of a pseudohalfplane-hypergraph also as pseudohalfplanes.

Using Lemma~\ref{prop:slack} on a hyperedge of a pseudohalfplane-hypergraph, we get the following.

\begin{prop}\label{fact:szukit}
Given a pseudohalfplane-hypergraph $\HH$, and a hyperedge $A$ of $\HH$, we can add a new hyperedge $A'$ contained completely in $A$ that contains all but one of the points of $A$, such that $\HH$ remains a pseudohalfplane-hypergraph.
\end{prop}

In the geometric setting this corresponds to the known and useful fact that given a pseudohalfplane arrangement and a finite set of points $A$ contained in the pseudohalfplane $H$, we can add a new pseudohalfplane $H'$ contained completely in $H$ that contains all but one of the points of $A$.

Now we show how to extend Theorem~\ref{thm:ABA} to pseudohalfplane arrangements, i.e., to the case when the points of $S$ {\em below} a line also define a hyperedge.

\begin{thm}\label{thm:primalpshp}
Given a finite set of points $S$ and a pseudohalfplane arrangement $\HH$, we can color $S$ with $k$ colors such that any pseudohalfplane in $\HH$ that contains at least $2k-1$ points of $S$ contains all $k$ colors. Equivalently, the vertices $S$ of a pseudohalfplane-hypergraph can be colored with $k$ colors such that any hyperedge containing at least $2k-1$ points contains all $k$ colors. 
\end{thm}

\begin{remark} The similar statement is not true for the union of two arbitrary ABA-free hypergraphs (instead of an ABA-free hypergraph and its complement), as the union of two arbitrary ABA-free hypergraphs might not be $2$-colorable, see \cite{P14} for such a construction.
\end{remark}

\begin{proof}[Proof of Theorem~\ref{thm:primalpshp}]\label{app:primalpshp}
Our proof is completely about the abstract setting, yet it translates naturally to the geometric setting, also the figures illustrate the geometric interpretations.

By definition there exists an ABA-free $\FF$ such that $\HH\subset \FF\cup \bar{\FF}$. Call $\UU=\HH\cap \FF$ the upsets and $\DD=\HH\cap \bar{\FF}$ the downsets, observe that both $\UU$ and $\DD$ are ABA-free. 

Further, the unskippable vertices of $\UU$ (resp.\ $\DD$) are called top (resp.\ bottom) vertices. The top and bottom vertices are called the unskippable vertices of $\HH$. Recall that by adding these unskippable vertices as one-element hyperedges to $\HH$, $\HH$ remains a pseudohalfplane-hypergraph, as we can extend $\FF$ and $\bar{\FF}$ with the appropriate hyperedge (this is a convenient way of thinking about top/bottom vertices in the geometric setting, as seen later in the figures).

\begin{obs}\label{obs:topindown} If $x$ is top and $X$ is a downset and $x\in X$, then $X$ contains all vertices that are bigger or all vertices that are smaller than $x$. The same holds if $x$ is bottom, $X$ is an upset and $x\in X$.
\end{obs}

\begin{lem}\label{lem:shallow}
If  $\HH$ is a containment-free pseudohalfplane-hypergraph, then any minimal hitting set of $\mathcal H$ that contains only unskippable vertices is $2$-shallow.
\end{lem}
\begin{proof} 
Let $R$ be a minimal hitting set of unskippable vertices. 
Suppose for a contradiction that $\{a,b,c\}\subset R\cap X$ and $a<b<c$ for some $X\in\HH$.
Without loss of generality, suppose that $b$ is top.
As $R$ is minimal, let $B$ be a set for which $B\cap R=\{b\}$.
From Observation~\ref{obs:topindown} it follows that $B$ is an upset.

First suppose that $X$ is an upset.
As $B\not\subset X$, take a $b_2\in B\setminus X$.
As $B$ and $X$ are both upsets and thus have the ABA-free property, we have $b_2<a$ or $c<b_2$.
Without loss of generality, we can suppose $c<b_2$.
If $c$ is top, $\{c\}$ and $B$ violate ABA-freeness. See Figure~\ref{fig:cnottop}.
If $c$ is bottom, then using Observation~\ref{obs:topindown}, $X$ contains all the vertices that are smaller than $c$.
Take a set $A\not\subset X$ for which $A\cap R=\{a\}$.
This set must contain an $a_2\in A\setminus X$ and so we must have $c<a_2$.
If $A$ is an upset, as it does not contain $b$ and recall $a<b<a_2$, $A$ and $\{b\}$ violate ABA-freeness. See Figure~\ref{fig:anotupset}.
If $A$ is a downset, as it does not contain $c$ and recall $a<c<a_2$, $A$ and $\{c\}$ violate ABA-freeness, both cases lead to a contradiction.

\begin{figure}
    \centering
    \subfloat[$c$ cannot be top]{\label{fig:cnottop}
        \includegraphics[width=0.4\textwidth]{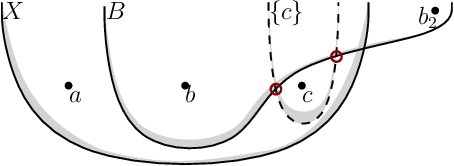}
				}    
    \hskip 20mm
    \subfloat[$A$ cannot be an upset]{\label{fig:anotupset}
        \includegraphics[width=0.4\textwidth]{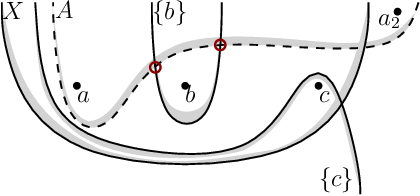}
				}
						\caption{Proof of Lemma~\ref{lem:shallow}}
\end{figure} 	 

The case when $X$ is a downset is similar.
Using Observation~\ref{obs:topindown} for $X$ and $\{b\}$ we can suppose without loss of generality that $X$ contains all vertices that are smaller than $b$.
Take a set $A\not\subset X$ for which $A\cap R=\{a\}$ and an $a_2\in A\setminus X$. As $X$ contains all vertices smaller than $b$, we have $b<a_2$.
$A$ cannot be an upset, as then it would contain $b$, so it is a downset.
If $b<a_2<c$, then $A$ and $X$ would violate ABA-freeness, thus we must have $c<a_2$.
This means $c$ cannot be bottom, so it is top.
Using Observation~\ref{obs:topindown}, $X$ contains all the vertices that are smaller than $c$.
But then $B\setminus X$ must have an element that is bigger than $c$, contradicting the ABA-freeness of $B$ and $\{c\}$.
\end{proof}



	

It is easy to see that being a pseudohalfplane hypergraph is a hereditary property. Thus, Lemma~\ref{lem:shallow} implies that all the assumptions of Theorem~\ref{algo:sy} hold with $c=2$ to get a polychromatic $k$-coloring as required. This finishes the proof of Theorem~\ref{thm:primalpshp}.
\end{proof}

\section{Dual problem and pseudohemisphere-hypergraphs}\label{sec:signed}

We are also interested in coloring pseudohalfplanes with $k$ colors such that all points that are covered many times will be contained in a pseudohalfplane of each $k$ colors.
For example, we can also generalize the dual result about coloring halfplanes of \cite{SY12} to pseudohalfplanes.

\begin{thm}\label{thm:dualpshp}
Given a pseudohalfplane arrangement $\HH$, we can color $\HH$ with $k$ colors such that
if a point $p$ belongs to a subset $\HH_p$ of at least $3k-2$ of the pseudohalfplanes of $\HH$, then
$\HH_p$ contains a pseudohalfplane of every color.
\end{thm}

Theorem~\ref{thm:dualpshp} follows from Theorem~\ref{thm:shitting}, that we will state and prove later.

However, instead of coloring pseudohalfplanes, we stick to coloring points with respect to pseudohalfplanes and work with {\em dual hypergraphs}, where the vertex-hyperedge incidences are preserved, but vertices become hyperedges and hyperedges become vertices.

\begin{prop}\label{prop:dual2X}
A hypergraph $\HH$ on an ordered set of vertices $S$ is a {\em dual pseudohalfplane-hypergraph} if and only if there exists a set $X\subset S$ and an ABA-free hypergraph $\cal F$ on $S$ such that the hyperedges of $\HH$ are the hyperedges $F\Delta X$ for every $F\in \cal F$ (where $\Delta$ denotes the symmetric difference of two sets).
\end{prop}
\begin{proof}
	Recall that pseudohalfplane-hypergraphs are hypergraphs that we can get by taking the complement of some hyperedges in an ABA-free hypergraph.\footnote{According to the definition we may need to duplicate some of the hyperedges so that we have both the original and its complement, but by duplicating some hyperedges the hypergraph remains ABA-free.} In relation to their incidence matrix, using Theorem $\ref{thm:matrix}$, this means that a hypergraph is a pseudohalfplane-hypergraph if and only if there is a permutation of the rows and columns of its incidence matrix such that inverting some of the columns (i.e., exchanging $0$'s and $1$'s in these columns) we get a matrix which is  $\begin{bmatrix} 
	0 & 1  \\
	1 & 0 
	\end{bmatrix}$-free. Taking the dual of such a hypergraph means taking the transpose of such an incidence matrix. 
	
	Thus a hypergraph $\HH$ is a dual pseudohalfplane-hypergraph if and only if there is a permutation of the rows and columns of its incidence matrix such that inverting some of the rows we get a matrix which is  $\begin{bmatrix} 
	0 & 1  \\
	1 & 0 
	\end{bmatrix}$-free. Using again Theorem $\ref{thm:matrix}$ we get that this is equivalent to the fact that the incidence matrix of $\HH$ is the incidence matrix of an ABA-free hypergraph with some of the rows inverted. Finally, this is equivalent to the statement of the proposition with $X$ being the subset of vertices corresponding to the inverted rows.
\end{proof}

Now we define a common generalization of the primal and dual definitions. 

\begin{defi}\label{def:signedpseudo}
A {\em pseudohemisphere-hypergraph} is a hypergraph $\HH$ on an ordered set of vertices $S$ such that there exists a set $X\subset S$ and an ABA-free hypergraph $\cal F$ on $S$ such that the hyperedges of $\HH$ are some subset of $\{F\Delta X, \bar F \Delta X \mid F\in \cal F\}$.
\end{defi}

\begin{prop}\label{prop:dualsignedpseudo}
	The  dual of a pseudohemisphere-hypergraph is also a pseudohemi\-sphere-hypergraph.
\end{prop}
\begin{proof}
	Notice that by definition a hypergraph $\HH$ is a pseudohemisphere-hypergraph if and only if some rows and columns of its incidence matrix can be inverted such that it becomes the incidence matrix of an ABA-free hypergraph. Using Theorem $\ref{thm:matrix}$ we get that this is equivalent to the fact that we can permute the rows and columns of the incidence matrix of $\HH$ and invert some of the rows and columns to get a $\begin{bmatrix} 
	0 & 1  \\
	1 & 0 
	\end{bmatrix}$-free matrix. This property obviously holds for a matrix if and only if it holds for its transpose and thus, similarly to Proposition \ref{prop:dual2X}, we can conclude that the dual of $\HH$ is also a pseudohemi\-sphere-hypergraph.	
\end{proof}


%
%

Furthermore, there is a nice geometric representation of such hypergraphs using {\em pseudohemisphere arrangements}, a generalization of hemisphere arrangements on a sphere.

In a pseudohemisphere arrangement the pseudohemispheres are regions whose boundaries are centrally symmetric simple curves such that any two intersect exactly twice.
(For more on pseudohemisphere arrangements, see, e.g., \cite{B99}.)
Without changing the combinatorial properties of the arrangement, we can suppose that the boundary of one of the pseudohemispheres is the equator.
Using a stereographic projection from the center of the sphere such that this pseudohemisphere is mapped to a whole plane, the other pseudohemispheres are mapped to pseudohalfplanes.
Thus, we can conclude that $\HH$ is a pseudohemisphere-hypergraph if and only if there is a set of points, $S$, on the surface of a sphere and a {\em pseudohemisphere arrangement} $\FF$ on the sphere such that the incidences among $S$ and $\FF$ give $\HH$.
(Here $X$ corresponds to the points on the southern hemisphere and $S\setminus X$ to the points on the northern hemisphere.)

Another popular geometric representation on the plane, adding {\em signs} to lines and points, is the following.
The vertices correspond to a set of points in the plane together with a direction (up or down), and the hyperedges correspond to a set of (x-monotone) pseudolines with a sign ($+$ or $-$).
The hyperedge corresponding to a positive pseudoline is the set of points that point {\em towards} the pseudoline, while the hyperedge corresponding to a negative pseudoline is the set of points that point {\em away} from the pseudoline.
Positive pseudolines correspond to $\FF$, negative pseudolines to $\bar \FF$, up points correspond to $X$ and down points correspond to $\bar X$.
With this interpretation, ABA-free hypergraphs have only + and up signs, pseudohalfplane-hypergraphs have $\pm$ and up signs, dual pseudohalfplane-hypergraphs have + and up/down signs.

In the next table we summarize the best known results about these hypergraphs, with respect to how many points each hyperedge has to contain to have a polychromatic $k$-coloring and the values of the smallest $c$ for which there exists a $c$-shallow hitting set for containment-free families.

\begin{table}[h]
\begin{tabular}{|l|c|c|ll}
\cline{1-3}
 & Polychromatic $k$-coloring & Shallow hitting set &  &  \\ \cline{1-3}
ABA-free hypergraphs & $2k-1 $ (Theorem~\ref{thm:ABA}) & $2$ (Lemma~\ref{lem:spshallow}) &  &  \\ \cline{1-3}
Pseudohalfplane-hypergraphs & $2k-1$ (Theorem~\ref{thm:primalpshp}) & $2$ (Lemma~\ref{lem:shallow}) &  &  \\ \cline{1-3}
Dual pseudohalfplane-hypergraphs & $\le 3k-2$ (Theorem~\ref{thm:dualpshp})& $\le 3$ (Theorem~\ref{thm:shitting}) &  &  \\ \cline{1-3}
Pseudohemisphere-hypergraphs & $\le 4k-3$ (Corollary~\ref{cor:sphere})& $\le 4$ (Theorem~\ref{thm:shitting}) &  &  \\ \cline{1-3}
\end{tabular}
\end{table}

We conjecture that even containment-free pseudohemisphere arrangements have a $2$-shallow hitting set, which would also imply, using Theorem~\ref{algo:sy}, that any family whose sets have size at least $2k-1$ admits a polychromatic $k$-coloring. Towards this conjecture, the only result not in the table is about the special case of dual (ordinary) halfplanes, for which Fulek \cite{F11} showed that in the $k=2$ case $2k-1=3$ is the right answer. That is, he showed that we can $2$-color any family of halfplanes such that every point of the plane which belongs to at least $3$ halfplanes is covered by halfplanes of both colors.

As we can find a polychromatic $k$-coloring of the points of $X$ and $\bar X$ independently with respect to the sets of $\FF$ and $\bar \FF$, respectively, of size at least $2k-1$ using Theorem~\ref{thm:primalpshp}, the following is true.

\begin{cor}\label{cor:sphere}
Given a finite set of points $S$ on the sphere and a pseudohemisphere arrangement $\HH$, we can color $S$ with $k$ colors such that any pseudohemisphere in $\HH$ that contains at least $4k-3$ points of $S$ contains all $k$ colors. Equivalently, the vertices $S$ of a pseudohemisphere-hypergraph can be colored with $k$ colors such that any hyperedge containing at least $4k-3$ points contains all $k$ colors. 
\end{cor}

To finish, we first prove the following theorem, 
which, using Theorem~\ref{algo:sy}, will imply Theorem~\ref{thm:dualpshp}, and also provides another proof for Corollary~\ref{cor:sphere}.

\begin{thm}\label{thm:shitting} Every containment-free dual pseudohalfplane-hypergraph has a $3$-shallow hitting set and every containment-free pseudohemisphere-hypergraph has a $4$-shallow hitting set.
\end{thm}

The proof of this result follows again closely the argument of \cite{SY12}. 
We note that the next few statements can also be proved using the geometric representation, but here we develop further our completely abstract approach.
The reason for this is to demonstrate the power of our method, hoping that in the future it enables attacking completely different problems as well.
For an ordered set of vertices $S=Y\cup^* Z$, write $S=(Y,Z)$ if the vertices in $Y$ precede the ones in $Z$. 

\begin{lem}\label{lem:pushbackABA}\label{lem:pushback}
Suppose $\FF$ is an ABA-free hypergraph on an ordered vertex set $S=(Y,Z)$.
Then $\FF'=\FF\Delta Y=\{F\Delta Y \mid F\in\FF\}$ is an ABA-free hypergraph on the vertices ordered as $S'=(Z,Y)$, i.e., $Z$ precedes $Y$ but otherwise the order inside $Y$ and $Z$ is unchanged. 

Moreover, if $\FF$ and $X\subset S$ define a pseudohemisphere-hypergraph $\HH$, i.e., the hyperedges of $\HH$ are $\{F\Delta X\mid F\in\FF\}$ and $\{\bar F\Delta X\mid F\in\FF\}$, then $\FF'$ and $X'=X\Delta Y\subset S'$ also define the same (if unordered) pseudohemisphere-hypergraph $\HH'$.
\end{lem} 
\begin{proof}
It is enough to show the statement if $|Y|=1$, as then by induction we can proceed with the vertices of $|Y|>1$ one by one.
Let us denote the original order by $<$ and the new one by $\prec$.
It is enough to show that for any $A,B\in \FF$ we have no ABA-sequence in $A'=A\Delta Y, B'=B\Delta Y\in \FF'$ according to the order $\prec$. 
We will only use that there is no ABA-sequence in $A,B$ according to $<$.
Denote the only element of $Y$ by $y$.
If $y\notin A\Delta B$, then $A\Delta B$ is unchanged by the transformation, thus an ABA-sequence in $A',B'$ according to $\prec$ would also be an ABA-sequence in $A,B$ according to $<$, a contradiction.
Thus, without loss of generality, $y\in B\setminus A$ and so $y\in A'\setminus B'$.
An ABA-sequence in $A',B'$ according to $\prec$ not containing $y$ would be an ABA-sequence also in $A,B$ according to $<$.
Otherwise, if three vertices $a\prec b\prec y$ form an ABA-sequence in $A',B'$, then the three vertices $y<a<b$ form an ABA-sequence in $B,A$, a contradiction.

For the moreover part, notice that as $F\Delta Y \Delta X'= F\Delta Y \Delta X\Delta Y=F\Delta X$, the hyperedges of $\HH$ and $\HH'$ are indeed the same.
\end{proof}

\begin{remark} Lemma~\ref{lem:pushbackABA} suggests that instead of our linear ordering of the vertices, we could consider them in circular order.
Indeed, let the vertices be points in a circle, where for every vertex the point opposite to it on the circle is also a vertex, called its negated pair.
Now take a hypergraph on such a circular point set which contains exactly one point from each opposite pair and is {\em circular ABAB-free}, that is, it does not contain two sets, $A$ and $B$, and four points, $a,b,c,d$, that are in this order around the circle for which $a,c\in A\setminus B$ and $b,d\in B\setminus A$.
It is easy to see that such a hypergraph is also circular ABABAB-free, and restricting it to any consecutive subset of half of the vertices is an ABA-free hypergraph with the same (non-circular) order. 
For example if the original base set is $S=(a,b,c)$ in this order and a set in the family is $F=\{a, c\}$, then in the circular order the base set  is $(a,b,c,\bar a, \bar b, 
\bar c)$ and $\{a, c, \bar b\}$ is  $F$. After we apply Lemma~\ref{lem:pushbackABA} with $Y=\{a\}$, we essentially rotate the non-circular base set by one in the circular order and the ``new'' base set becomes $S'=\{b,c,\bar a\}$. In the circular order $F$ is still $\{a, c, \bar b\}=\{\bar b, c, \bar {\bar a}\}$ which is $\{c\}$ over $S'$ (as only $c$ is non-negated compared to $S'$).

Our earlier results could be translated to this abstraction as well, which models the above rotational symmetry of pseudohemispheres in a more natural way.
However, further statements we prove are still non-trivial even in this model, so we will stick with our original linear ordering of the vertices.
\end{remark}

\begin{lem}\label{lem:helly}[Helly's theorem for pseudohalfplanes]
If any three hyperedges of a pseudohalf\-plane-hypergraph intersect, then we can add a vertex contained in all pseudohalfplanes of the arrangement.
\end{lem}
\begin{proof}
We prove the dual statement, as it will be more convenient.
That is, suppose that we are given a pseudohemisphere-hypergraph $\HH$, such that all its hyperedges are derived from $\FF$, i.e., $\HH$ has a representing ABA-free $\FF$ and vertex set $X\subset S$ such that for every $H\in\HH$ there is an $F\in \FF$ such that $H=F\Delta X$. 
We need to show that if for any three vertices there exists a hyperedge that contains all three of them, then we can add the hyperedge $\bar{X}$ to $\FF$ such that it stays ABA-free.
This is indeed the dual equivalent of the statement, as $\bar{X}\Delta X\in\HH$ contains all the vertices.

For a contradiction, suppose that $\bar X$ and some $F\in \FF$ violate ABA-freeness because of some vertices $x,y,z$.
By our assumption, there exists another hyperedge $G\Delta X$ which contains all of $x,y,z$, thus $G$ and $\bar X$ contain the same subset of $x,y,z$.
Thus $F,G\in \FF$ contain an ABA-sequence on the vertices $x,y,z$ as $F,\bar X$ contain an ABA-sequence on $x,y,z$, a contradiction. 
\end{proof}

Applying this to the complements of the pseudohalfplanes we get the following.

\begin{cor}\label{cor:hellycor}
Given a pseudohalfplane-hypergraph, either there are already three hyperedges that cover all the vertices, or we can add a vertex which is in none of the hyperedges. 
\end{cor}

Now we show that reordering the vertices in an appropriate way keeps the ordered hypergraph ABA-free.

\begin{lem}\label{lem:pushfront}
Suppose $\FF$ is an ordered ABA-free hypergraph on vertex set $S$.
Let $F\in \FF$ be a smallest hyperedge in the partial ordering of the hyperedges of $\FF$.
If we reorder $S$ as $(F,\bar F)$, i.e., the vertices of $F$ go to the front but otherwise the order inside $F$ and $\bar F$ is unchanged, then the ordered hypergraph remains ABA-free. 
\end{lem}
\begin{proof}
Let us denote the original order by $<$ and the new one by $\prec$.
Suppose on the contrary, that for some $A,B\in\FF$ we have some $a, c\in A\setminus B$ and $b\in B\setminus A$ that satisfy $a\prec b\prec c$.
The proof is a simple case analysis of how this could happen.
Notice that $c\in F$ implies $b\in F$ and $b\in F$ implies $a\in F$, so there are four cases.
If $a,b,c\in F$ or $a,b,c\notin F$, then $a<b<c$.
In this case $A$ and $B$ contradict that $\FF$ is ABA-free.
If $a\in F$ and $b,c\notin F$, then we must have $b<a$.
In this case $B<F$, contradicting that $F$ is smallest.
If $a,b\in F$ and $c\notin F$, then we must have $c<b$.
In this case $A<F$, contradicting that $F$ is smallest.
\end{proof}

\begin{remark} If $S=\{a<b<c\}$ and $\FF=\{\{a\}, \{c\}, \{a,b\}, \{a,c\}, \{b,c\}\}$, then in any reordering of $S$ where the elements of the hyperedge $\{a,c\}$ go to front (i.e., in $\{a<c<b\}$ and $\{c<a<b\}$) ABA-freeness is violated.
This shows that in the above Lemma~\ref{lem:pushfront} the assumption that $F$ is a smallest hyperedge cannot be removed.
We might hope that the lemma can be modified to remain true for all hyperedges by first applying Lemma~\ref{lem:pushback} for an appropriate prefix set of the points, however this is also not possible.
Consider the ABA-free hypergraph $\FF=\{\emptyset, \{a\}, \{c\}, \{a,b\}, \{a,c\}, \{b,c\}, \{a,b,c\}\}$ and define $\FF_X=\{F\Delta X\mid F\in \FF\}$ for any $X\subset S=\{a,b,c\}$.
In this case there is no $X$ for which there is a reordering of $S$ that starts with the elements of $\{a,c\}\Delta X$ and for which $\FF_X$ is ABA-free with this new order.
\end{remark}

\begin{lem}\label{lem:polar}
If all the hyperedges of a pseudohemisphere-hypergraph $\HH$ avoid some vertex $p$ in $S$, then $\hat\HH$, the dual hypergraph of $\HH$, is a pseudo\-half\-plane-hypergraph.
\end{lem}
\begin{proof}
Start with a representation of $\HH$: an ABA-free hypergraph $\FF$ and a point set $X$ such that $\HH \subset \{F\Delta X, \bar F\Delta X \mid F\in \FF\}$.
Apply Lemma~\ref{lem:pushback} with $Y$ being the vertices before $p$, this way we get a representation of $\HH$ in which $p$ is the first point.
Take $\hat\HH$, the dual of $\HH$, with representation $\hat\FF$ and $\hat X$.
In $\hat\HH$, the set corresponding to $p$ is $H_p=F_p\Delta \hat X$ for some $F_p\in \hat \FF$, where we can choose the representation such that $F_p$ is the smallest set of $\hat \FF$ (because of the ordering used in Proposition~\ref{prop:dual}, as $p$ was the smallest point of $\FF$).
Now apply Lemma~\ref{lem:pushfront} to get another representation of $\hat \HH$ in which the points of $F_p$ are at the beginning in the order.
As $p$ was a point that was in none of the hyperedges of $\HH$, in the dual $H_p$ contains no points and so $F_p=H_p\Delta \hat X=\emptyset\Delta \hat X=\hat X$.
Now apply again Lemma~\ref{lem:pushback} to $\hat \FF$ with $Y=\hat X$.
We get a representation $(\hat{\cal F}',\hat X')$ of $\hat \HH$ in which $\hat{X'}=\hat X\Delta \hat X=\emptyset$, that is, $\hat \HH$ is a pseudohalfplane-hypergraph.
\end{proof}

Applying Lemma \ref{lem:polar} to the dual of a pseudohemisphere-hypergraph we get the following dual statement:

\begin{cor}\label{cor:polar}
If the empty set is (or can be added as) a hyperedge of a pseudohemisphere-hypergraph $\HH$, then $\HH$ is a pseudo\-half\-plane-hypergraph.
\end{cor}

\begin{lem}\label{lem:helly2}[Helly's theorem for pseudohemispheres]
If any four hyperedges of a pseudohemi\-sphere-hypergraph intersect, then we can add a vertex contained in all pseudohemispheres of the arrangement.
\end{lem}
\begin{proof}
Let $\HH$ be defined by $\cal F$ and $X\subset S$.
We prove the following stronger statement.
If there is a pseudohemisphere $F_0\Delta X=H_0\in \HH$ that has a non-empty intersection with any three other pseudohemispheres, then we can add a vertex contained in all the pseudohemispheres of the arrangement. 
Let $X'=\bar F_0=S\setminus F_0$ and denote by $\HH'$ the pseudohemisphere-hypergraph defined on $S$ by $\cal F$ and  $X'$.
As $H_0'=F_0\Delta X'=F_0\Delta (S\setminus F_0)=S$ contains all the points, we can apply Corollary~\ref{cor:polar} to $\HH'$ and the complement of $H_0'$ to conclude that $\HH'$ is a pseudohalfplane-hypergraph.
It follows from our definitions that the hyperedges in $\HH$ and $\HH'$ are in bijection such that for every $H\in \HH$ there is an $G\in \HH'$ (and vice versa) such that $H=G\Delta X'\Delta X$.

Next, we prove that in $\HH'$ any three pseudohemispheres intersect.
Suppose that the original intersection point of these pseudohemispheres with $H_0$ in $\HH$ was some $p\in H_0 \cap H_1 \cap H_2 \cap H_3$, where $H_i=F_i\Delta X$ for $1\le i \le 3$.
This implies $p\in (F_0\setminus X) \subset (S\setminus X')$ or $p\in (X\setminus F_0) \subset X'$.
In the first case, $p\in F_i$ and $p\in F_i\Delta X'=H_i'$.
In the second case, $p\notin F_i$ and $p\in F_i\Delta X'=H_i'$.

Therefore, any three pseudohalfplanes of $\HH'$ intersect.
Using Lemma~\ref{lem:helly} for the pseudo\-half\-plane-hyper\-graph representation of $\HH'$, we can add a new point $q$ to all the hyperedges of $\HH'$.
Denote this new pseudohalfplane-hypergraph by $\HH^+$, and let $S^+=S\cup \{q\}$ and $X^+=X'\Delta X$ (note that $q\notin X^+$).
The hypergraph $\HH'^{+}=\{G^+\Delta X^+ \mid G^+\in \HH^+\}$ on the base set $S^+$ is also a pseudohemisphere-hypergraph.
Moreover, we claim that it is the same as $\{H\cup \{q\} \mid H\in \HH\}$, which proves the lemma.
Indeed, recall that each hyperedge $H\in \HH$ is in bijection with a hyperedge $G\in \HH'$ with $H=G\Delta X'\Delta X=G\Delta X^+$.
Thus, each hyperedge $H^+=G\cup  \{q\}\in \HH^+$ is in bijection with the corresponding $H\cup \{q\}=G\Delta X^+$.
This implies that $\HH'^+=\{G\cup \{q\}\Delta X^+: G\in \HH'\}$=\{$H\cup \{q\}:H\in \HH\}$.
\end{proof}

Applying this to the complements of the pseudohemispheres we get the following.

\begin{cor}\label{cor:hellycor2}
Given a pseudohemisphere-hypergraph, either there are four hyperedges that cover all the vertices, or we can add a vertex which is in none of the hyperedges. 
\end{cor}

Now we are ready to prove Theorem~\ref{thm:shitting}. 

\begin{proof}[Proof of Theorem~\ref{thm:shitting}]
First we prove that every containment-free dual pseudohalfplane-hypergraph $\cal H$ has a $3$-shallow hitting set. Consider the dual of $\cal H$, the pseudohalfplane-hypergraph $\hat {\cal H}$.

If in  $\hat {\cal H}$ there is a set of at most $3$ hyperedges covering every point, then in $\cal H$ the corresponding $3$ vertices form a $3$-shallow hitting set.
Otherwise, by Corollary~\ref{cor:hellycor} we could add a point to $\hat{\cal H}$ that is in none of the pseudohalfplanes.
In this case, by Lemma~\ref{lem:polar} the dual of $\hat {\cal H}$, which is actually $\cal H$ itself, is a pseudohalfplane-hypergraph (note that we do not include the empty hyperedge that would be the dual of the newly added point).
By Lemma~\ref{lem:shallow} it has a $2$-shallow hitting set, which is also a $3$-shallow hitting set.
This finishes the proof of the first statement of Theorem~\ref{thm:shitting}.

Now we can similarly prove that every containment-free pseudohemisphere-hypergraph has a $4$-shallow hitting set. Let $\cal H$ be this hypergraph and take again its dual, $\hat {\cal H}$. If there is a set of at most $4$ hyperedges covering every point in $\hat{\cal H}$, then in $\cal H$ the corresponding $4$ vertices form a $4$-shallow hitting set.
Otherwise, by Corollary~\ref{cor:hellycor2} we could again add a point to $\hat {\cal H}$ that is in none of the pseudohalfplanes. As before this and Lemma~\ref{lem:polar} imply that $\cal H$ is a pseudohalfplane-hypergraph and thus by Lemma~\ref{lem:shallow} it has a $2$-shallow hitting set.
\end{proof}

\begin{proof}[Proof of Theorem~\ref{thm:dualpshp} and of Corollary~\ref{cor:sphere}]
Being a dual pseudohalfplane hypergraph and being a pseudohemisphere hypergraph are hereditary properties.
Thus, Theorem~\ref{thm:shitting} implies that all the assumptions of Theorem~\ref{algo:sy} hold  with $c=3$ and $c=4$, respectively, to get the polychromatic colorings required.
\end{proof}

\section{ABAB-free hypergraphs 
and more}\label{sec:bless}

Definition~\ref{def:ABA} can be generalized in a straightforward way, similarly to Davenport-Schinzel sequences \cite{DS}, to more alternations.
Our goal in this section is to show that already one more alternation gives non-two-colorable hypergraphs.

\begin{defi}\label{def:ABAB}
A hypergraph $\mathcal H$ with an ordered vertex set is called {\em ABAB-free} if $H$ does not contain two hyperedges $A$ and $B$ for which there are four vertices $w<x<y<z$ such that $w,y\in A\setminus B$ and $x,z\in B\setminus A$.

A hypergraph with an unordered vertex set is ABAB-free if its vertices have an ordering with which the hypergraph is ABAB-free.
\end{defi}

We remark that similarly to Proposition~\ref{prop:abaequipseudoline}, an ABAB-free hypergraph corresponds to an arrangement of graphic curves that intersect at most twice.

\subsection{ABAB-free hypergraphs that are not two-colorable}
We show that there are ABAB-free hypergraphs that do not have a proper $2$-coloring.
We prove this by ordering the vertices of a non-$2$-colorable hypergraph $\HH_k$ in a tricky way to give an ABAB-free hypergraph.
First we define this hypergraph $\HH_k$ often used in counterexamples, e.g., \cite{PTT09}.

\begin{figure}[t]
    \centering
        \includegraphics[scale=0.7]{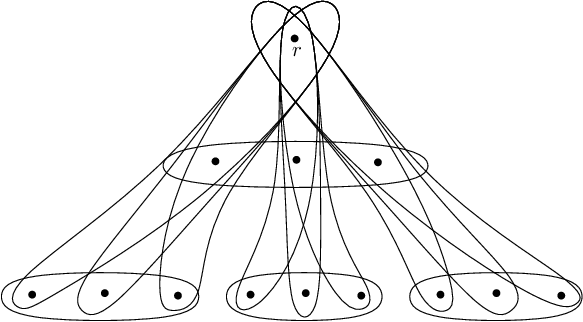}
   \caption{$\HH_3$}
   \label{fig1}
\end{figure}

\begin{figure}[t]
    \centering
        \includegraphics[scale=0.68]{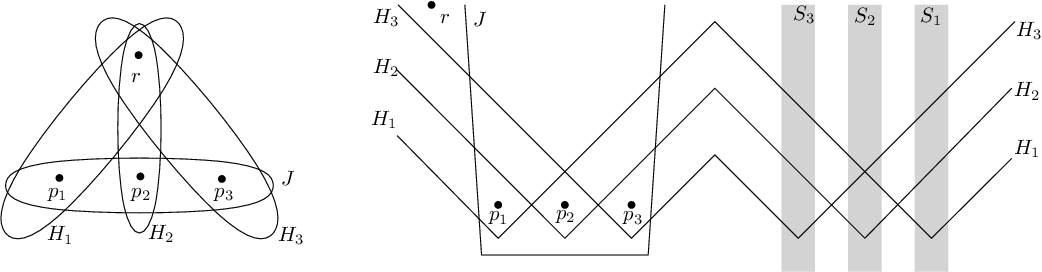}
   \caption{$\HH'_2$ and its realization with pseudoparabolas (for $k=3$)}
   \label{fig2}
\end{figure}

\begin{defi}
Let $G_k$ be the complete $k$-ary tree of depth $k$, i.e., the rooted tree such that its root $r$ has $k$ children, each vertex of $G_k$ in distance at most $k-2$ from $r$ has $k$ children and the vertices in distance $k-1$ from $r$ are the leafs (without children).

$\HH_k$ is the $k$-uniform hypergraph which has two types of hyperedges.
First, for every non-leaf vertex the set of its children form an hyperedge.
Second, the vertices of every descending path starting in $r$ and ending in a leaf form an hyperedge.
\end{defi}

It is easy to see that $\HH_k$ is not two-colorable. 
Now we show how to realize $\HH_k$ such that its vertices correspond to points in the plane and its hyperedges correspond to the points {\em above} pseudoparabolas (simple curves such that any two intersect at most {\em twice}). This implies that the $x$-coordinates define an ordering of the vertices of $\HH_k$ showing that $\HH_k$ is ABAB-free.
We fix $k$ and define $\HH'_l$ (resp.\ $G'_l$) to be the hypergraph (resp.\ graph) induced by $\HH_k$ (resp.\ $G_k$) and the subset of the vertices that are in distance at most $l-1$ from the root $r$ in $G_k$ ($\HH'_l$ is a simple hypergraph, i.e., if multiple hyperedges induce the same hyperedge, we take it only once). Thus in particular $G'_1$ has one vertex and no hyperedges while $\HH'_1$ has one vertex and one hyperedge containing it, while $\HH'_k=\HH_k$ and $G'_k=G_k$. Note that in $G'_l$ every non-leaf vertex has $k$ children, and $\HH'_l$ has hyperedges of size $l$ corresponding to descending paths (which we usually denote by $H_i$ for some $i$) and hyperedges of size $k$ corresponding to the set of children of some vertex (which we usually denote by $J_i$ for some $i$). See Figure~\ref{fig1}. 

In our realization, to simplify the presentation, points corresponding to vertices will be denoted with the same label, and similarly hyperedges and the corresponding pseudoparabolas will have the same label.

\begin{figure}[t]
    \centering
        \includegraphics[width=\textwidth]{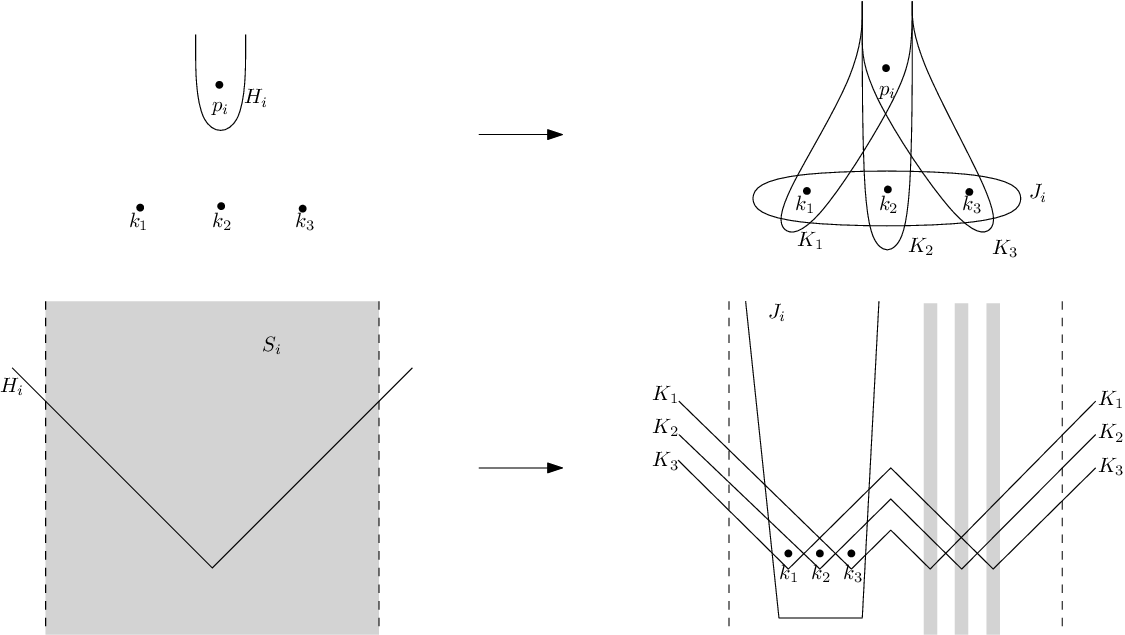}
   \caption{Recursive realization of $\HH'_l$: adding $k$ children to a leaf}
   \label{fig3}
\end{figure}

We will recursively realize $\HH'_l$, for an illustration see Figure~\ref{fig3}. We additionally maintain that each hyperedge (pseudoparabola) $H_i$ corresponding to a descending path has a vertical strip $S_i$ associated to it, such that inside $S_i$ there are no points and $H_i$ has the lowest boundary (thus no other hyperedge intersects $H_i$ inside $S_i$).
For $l=1$, this is trivial to do as $\HH'_1$ has one vertex and one hyperedge containing this vertex.
For $l=2$, Figure~\ref{fig2} shows a way to achieve this (for $k=3$).
Now suppose that for some $l$ we have $\HH'_l$ and we want to construct $\HH'_{l+1}$.
Take the construction of $\HH'_l$, and for each hyperedge $H_i$ corresponding to a descending path $P_i$ with endvertex $p_i$, do the following.
First make $k$ vertically translated copies of $H_i$ very close to each other.
Denote these by $K_1,K_2,\dots K_k$.
Next, using these $k$ copies of $H_i$, realize $\HH'_2$ (except the root $r$) in an appropriately small area inside $S_i$, by adding $k$ more points $k_1,k_2,\dots k_k$ such that for every $i$, $k_i$ is above $K_i$ and below every other pseudoparabola.
These points correspond to the children of $p_i$.
Finally, define the pseudoparabola $J_i$, which corresponds to the hyperedge containing all the $k_i$'s but no other vertex, as a parabola very close to the vertical strip containing the $k_i$'s.
For each $i$, the vertical strip that belongs to $K_i$ in the inner copy of $\HH'_2$ is the strip corresponding to the descending hyperedge that ends at $k_i$.
Therefore all properties are maintained, and by repeating the above procedure for each of the leafs $p_i$ of $\HH'_l$ we get a realization of $\HH'_{l+1}$.\qed

We are not aware of any nice characterization for the dual of ABAB-free hypergraphs, like we had for ABA-free hypergraphs in Proposition~\ref{prop:dual}.

\subsection{Bottomless rectangles and balanced colorings}

Every hypergraph given by a set of points and a collection of bottomless rectangles is ABAB-free, but not necessarily ABA-free.
In fact, it is not hard to see that such hypergraphs would correspond exactly to ``aBAb''-free hypergraphs, which can be defined similarly to Definition~\ref{def:ABA} as follows.

\begin{defi}\label{def:abab}
A hypergraph whose vertices are real numbers is {\em aBAb-free} if for any two of its hyperedges, $A$ and $B$, and vertices $x_1<x_2<x_3<x_4$ it does {\em not} hold that
$x_1\in A$, $x_2\in B\setminus A$, $x_3\in A\setminus B$, $x_4\in B$.
\end{defi}
 
It was shown in \cite{A13} that any finite set of points can be colored with $k$ colors such that any bottomless rectangle that contains at least $3k-2$ points contains a point of every color.
Unfortunately, we were not able to prove this using our methods, because containment-free bottomless rectangle families do not have a shallow hitting set, as shown by the following example.

\begin{figure}[t]
    \centering
        \includegraphics[scale=1]{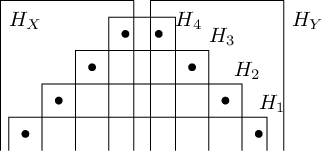}
   \caption{A containment-free bottomless rectangle family without a shallow hitting set}
   \label{fig:blessnotshallow}	
\end{figure}

\begin{ex}\label{ex:shallow}
Consider the set of points $X=\{(i,i)\mid i=1..k\}$ and $Y=\{(k+i,k+1-i)\mid i=1..k\}$ and the bottomless rectangle family that consists of the following.
\begin{enumerate}
\item A rectangle $H_X$ containing $X$.
\item A rectangle $H_Y$ containing $Y$.
\item Rectangles $H_i$ containing $(i,i)$ and $(2k+1-i,i)$ for $i=1..k$.
\end{enumerate}
Any hitting set for the $H_i$ rectangles contains $k/2$ points from $X$ or $Y$, thus it is not $(k/2-1)$-shallow for $H_X$ or $H_Y$ (for an illustration for $k=4$ see Figure~\ref{fig:blessnotshallow}).
\end{ex}

Instead of shallow hitting sets, we can ask whether a $k$-coloring exists for any contain\-ment-free bottomless rectangle family that satisfies a certain nice property, that can be achieved by repeatedly finding $c$-shallow hitting sets and making each of them a separate color class.
In the proofs in earlier sections, after $k$ shallow hitting sets were found and colored to different colors, we did not care about the remaining points, they were colored arbitrarily.
Instead, we could find a $(k+1)$-st shallow hitting set for the remaining points and use the first color for them,
then the second color for the $(k+2)$-nd shallow hitting set, and so on, until there are no more points left.
In general in the $i$-th step the shallow hitting set is colored with color $i$ (mod $k$), where color $0$ and color $k$ denote the same color.
This way we achieve a coloring that is not just polychromatic, but also has the following {\em balanced} property.

\begin{defi}\label{def:balanced}
We say that a $k$-coloring is {\em $c$-balanced} if for any given set (hyperedge) of our family denoting the sizes of any two color classes in it by $n_1$ and $n_2$, then we have $n_1\le c(n_2+1)$.
\end{defi}

As we have seen above, if a family has a $c$-shallow hitting set, then it also has a $c$-balanced $k$-coloring for any $k$.
For uniform families, a converse also holds; if every set has size $n$, then any color class of a $c$-balanced $n/c$-coloring is a $c^2$-shallow hitting set.
For non-uniform families, however, these notions might differ, so it is natural to ask the following.

\begin{prob}\label{prob:balanced}
Is there a balanced coloring for any family of bottomless rectangles?
\end{prob}

Example~\ref{ex:shallow} generalizes easily to other families, such as the translates or homothets of a convex polygon, so there is not much hope to achieve shallow hitting sets for other interesting planar families.
We do not, however, know whether a balanced coloring exists for the above families.

\subsection*{Acknowledgement}
We would like to thank our anonymous referees 
for their several useful suggestions and comments.

\appendix

\section{Simple facts about pseudolines}\label{app:facts}
Here we list some well-known facts about pseudoline arrangements.

A curve is {\em graphic} if it is the graph of a function, i.e., an $x$-monotone infinite curve that intersects every vertical line of the plane. A {\em graphic} pseudoline arrangement is such that every curve is graphic.
We say that two pseudoline arrangements are {\em equivalent} if there is a bijection between their pseudolines such that the order in which a pseudoline intersects the other pseudolines remains the same.
A {\em pseudohalfplane arrangement} is a pseudoline arrangement, with a side of each pseudoline selected. 

\medskip
\noindent
{\bf Facts about pseudoline arrangements}
\renewcommand{\theenumi}{\Roman{enumi}}%
\begin{enumerate}
\item (Levi Enlargement Lemma) Given a pseudoline arrangement, any two points of the plane can be connected by a new pseudoline (if they are not connected already).
\item Given a pseudoline arrangement, we can find a pseudoline arrangement in which every pair of pseudolines intersects exactly once, and the order in which a pseudoline intersects the other pseudolines remains the same (ignoring the new intersections).
\item Given a pseudoline arrangement, we can find an equivalent graphic pseudoline arrangement.
\end{enumerate}

From these facts it follows that in the definition of a pseudohalfplane we can (and will) suppose that the underlying pseudoline arrangement is a graphic pseudoline arrangement.

Notice that ABA-free hypergraphs are in a natural bijection with (graphic) pseudoline arrangements and sets of points, such that each hyperedge corresponds to the subset of points {\em above} a pseudoline.

\begin{prop}\label{prop:abaequipseudoline}
Given in the plane a set of points $S$ (with all different $x$-coordinates) and a graphic pseudoline arrangement $L$, define the hypergraph ${\cal H}_{S,L}$ with vertex set $S$ such that for each pseudoline $l\in L$ the set of points above $l$ is a hyperedge of ${\cal H}_{S,L}$. Then ${\cal H}_{S,L}$ is ABA-free with the order on the vertices defined by the $x$-coordinates.

Conversely, given an ABA-free hypergraph $\cal H$, there exists a set of points $S$ and a graphic pseudoline arrangement $L$ such that ${\cal H}={\cal H}_{S,L}$.
\end{prop}
\begin{proof}
\begin{figure}
    \centering
        \includegraphics[width=\textwidth]{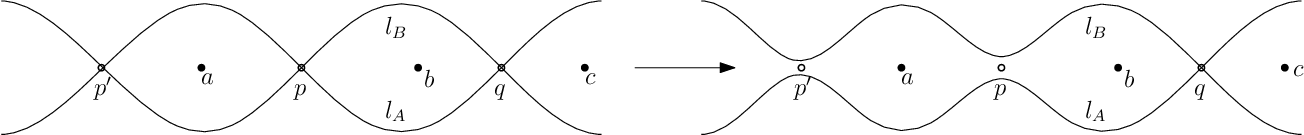}
   \caption{Redrawing a lens to decrease the number of intersections}
   \label{fig:lens}	
\end{figure}

The first part is almost trivial, suppose that there are two hyperedges $A,B$ in ${\cal H}_{S,L}$ having an ABA-sequence on the vertices corresponding to the points $a,b,c\in S$. The pseudolines corresponding to the hyperedges $A$ and $B$ are denoted by $\ell_A$ and $\ell_B$. The pseudoline $\ell_A$ intersects the vertical line through $a$ below $a$, the vertical line through $b$ above $b$ and the vertical line through $c$ above $c$, while $\ell_A$ intersects these in the opposite way (above/below/above). Thus these lines must intersect in the vertical strip between $a$ and $b$ and also in the strip between $b$ and $c$, thus having two intersections, a contradiction.

The second part of the proof is also quite natural. Given an ABA-free hypergraph ${\cal H}(V,E)$ with an ordering on $V$, we want to realize it with a planar point set $S$ and a graphic pseudoline arrangement $L$. 
Let $S$ be $|V|$ points on the $x$ axis corresponding to the vertices in $V$ such that the order on $V$ is the same as the order given by the $x$-coordinates on $S$. From now on we identify the vertices of $V$with the corresponding points of $V$.

For a given $A\in \cal H$ it is easy to draw an $\ell_A$ graphic curve for which the points of $S$ above $\ell_A$ are exactly in $A$. Draw a pseudoline $\ell_A$ for every $A\in \cal H$, such that there are finitely many intersections among these pseudolines, all of them crossings. What we get is an arrangement of graphic curves, but it can happen that they intersect more than twice. Now among such drawings take one which has the minimal number of intersections, we claim that this is a pseudoline arrangement. 

Assume on the contrary, that there are two curves $\ell_A$ and $\ell_B$ intersecting (at least) twice. Let two consecutive (in the $x$-order) intersection points be $p$ and $q$, where $p$ has smaller $x$-coordinate than $q$. Without loss of generality, $\ell_A$ is above $\ell_B$ close to the left of $p$ and close to the right of $q$, while $\ell_A$ is below $\ell_B$ in the open vertical strip between $p$ and $q$. This structure is usually called a lens, and we want to eliminate it in a standard way, decreasing the number of intersections. We can change the part of $\ell_A$ and $\ell_B$ to the left of $p$ (and to the right from the intersection $p'$ next to and left of $p$ if there is any) and change their drawing locally around $p$ (and $p'$ if it exists) such that we get rid of the intersection at $p$, see Figure~\ref{fig:lens}. If there are no points of $S$ between $\ell_A$ and $\ell_B$ and to the left of $p$ (and to the right of $p'$), then this redrawing does not change the hyperedges defined by $\ell_A$ and $\ell_B$, so we get a representation of $\cal H$ with less intersections, a contradiction. Thus there is a point $(p'<)a<p$ below $\ell_A$ and above $\ell_B$. Similarly, there must be a point $p<b<q$ above $\ell_A$ and below $\ell_B$ and finally a point $q<c$ below $\ell_A$ and above $\ell_B$, otherwise we could redraw the pseudolines with less intersections. These three points $a<b<c$ contradict the ABA-freeness of $\cal H$ as by the definition of the pseudolines, $b\in A\setminus B$ and $a,c\in B\setminus A$.
\end{proof}


\section{Small epsilon-nets for pseudohalfplanes}\label{app:epsnet}
Here we briefly mention the consequences of our results to $\epsilon$-nets of hypergraphs defined by pseudohalfplanes.
We omit proofs as they are not hard and can be obtained exactly as the corresponding results in \cite{SY12}.

Let ${\cal H} = (V, E)$ be a hypergraph where $V$ is a finite set. Let $\epsilon \in (0,1]$ be a real number.
A subset $N \subseteq V$ is called an $\epsilon$-net if for every hyperedge $S\in E$ such that $|S| \ge \epsilon |V |$, we also have $S \cap N \ne \emptyset$, i.e., $N$ is a hitting set for all ``large'' hyperedges.
It is known that hypergraphs with VC-dimension $d$ have small $\epsilon$-nets (of size $O(d/\epsilon\log(1/\epsilon))$ \cite{HW} and in general this is best possible \cite{Komlos}.
However, for geometric hypergraphs this is usually not optimal, in particular for halfplanes the following is true.
Consider a hypergraph ${\cal H} = (P,E)$ where $P$ is a finite set of points in the plane and $E=\{P\cap H \mid H \textit{ is a halfplane}\}$.
For this hypergraph there is an $\epsilon$-net of size $2/\epsilon-1$ for every $\epsilon$ \cite{Woeginger,SY12}.
Theorem~\ref{thm:primalpshp} implies that the same bound holds if the hypergraph is defined by pseudohalfplanes instead of halfplanes.
Also, for the dual hypergraph $\bar \HH$,
Theorem~\ref{thm:dualpshp} implies that there exists an $\epsilon$-net of size $3/ \epsilon$. Note that our results are in fact stronger as in the appropriate polychromatic coloring each color class intersects all large enough hyperedges, thus we get a partition of the vertices into $\epsilon$-nets (and at least one of them is a small $\epsilon$-net by the pigeonhole principle).

\end{document}